\documentclass{article}

\usepackage{epsfig}
\usepackage{amsfonts}
\usepackage{amssymb}
\usepackage{amsmath}   
\usepackage{amsthm}
\usepackage{latexsym}
\usepackage{enumerate}
\usepackage[subnum]{cases}
\usepackage{hyperref}
\usepackage{appendix}

\newtheorem{thm}{Theorem}[section]
\newtheorem{lemma}[thm]{Lemma}
\newtheorem{cor}[thm]{Corollary}
\newtheorem{prop}[thm]{Proposition}
\newtheorem{defin}[thm]{Definition}

\newtheorem{rmk}[thm]{Remark}

\numberwithin{equation}{section}

\DeclareMathOperator*{\dist}{dist}
\DeclareMathOperator*{\diam}{diam}

\renewcommand{\epsilon}{\varepsilon}

\newcommand{\R}{\mathbb{R}}

\newcommand{\vol}{\textup{vol}}

\newcommand{\marginals}{\Pi(\mu, \nu)}

\renewcommand{\hbar}{\bar{h}}

\newcommand{\inner}[2]{\ensuremath{\langle#1, #2\rangle}}

\newcommand{\Dx}[1][k]{\ensuremath{D_{x_#1}}}

\newcommand{\pdiff}[2][x]{\ensuremath{\frac{\partial}{\partial#1_{#2}}}}

\renewcommand{\S}{\mathbb{S}}
\newcommand{\cost}[2]{\frac{\lvert #1-#2\rvert^2}{2}}
\newcommand{\MTW}[4]{\left(MTW\right)^{#3#4}_{#1#2}}
\newcommand{\MTWcoord}[4]{-(c_{#1#2, pq}-c_{#1#2, r}c^{r, s}c_{s, pq})c^{p, #3}c^{q, #4}}
\newcommand{\norm}[1]{\lvert#1\rvert}
\newcommand{\Norm}[1]{\lVert#1\rVert}
\newcommand{\Sdiff}[2][]{D_{#1}{#2}}

\renewcommand{\dist}[2]{\inner{#1}{#2}}

\newcommand{\Jact}[2]{\frac{\partial T^i}{\partial x_j}}
\newcommand{\StimesS}{\S^n\times \S^n}
\newcommand{\factor}[1]{\beta\left(#1\right)}
\newcommand{\factorsub}[2]{\beta_{#2}\left(#1\right)}
\newcommand{\factornon}{\beta}
\newcommand{\factorsubnon}[1]{\beta_{#1}}
\newcommand{\tansp}[2]{T_{#1}{#2}}
\newcommand{\cotansp}[2]{T^\ast_{#1}{#2}}
\newcommand{\density}[2]{{e}^{#1#2}\sqrt{\det{{\g}_{ij}#2}}}
\newcommand{\rhoone}[1][]{\density{f}{#1}}
\newcommand{\rhobar}[1][]{\density{g}{#1}}
\newcommand{\rhoonet}[1][]{\density{f_t}{#1}}
\newcommand{\rhobart}[1][]{\density{g_t}{#1}}
\newcommand{\ut}{{u_t}}
\newcommand{\vt}{{v_t}}

\newcommand{\Tplus}[1][]{T^+_{#1}}
\newcommand{\Yplus}{Y^+}
\newcommand{\Tminus}[1][]{T^-_{#1}}
\newcommand{\Yminus}{Y^-}
\newcommand{\csubdiff}{\partial_c}

\newcommand{\yminus}{y^-}

\newcommand{\Tplusinv}[1][]{(\Tplus[#1])^{-1}}

\newcommand{\calD}[2]{\left(\funcnorm{\Sdiff{#1}}+ \funcnorm{\Sdiff{#2}}\right)}
\newcommand{\calX}{\mathcal{X}}
\newcommand{\calY}{\mathcal{Y}}
\newcommand{\gnorm}[1]{\norm{#1}_{\S^n}}
\newcommand{\funcnorm}[1]{\Norm{#1}}
\newcommand{\g}{{\mathring{g}}}

\newcommand{\lambdag}{\min{e^g}}

\newcommand{\Lambdaf}{\max{e^f}}
\newcommand{\lambdagt}{\min{e^{g_t}}}

\newcommand{\Lambdaft}{\max{e^{f_t}}}

\newcommand{\thmconst}{{\omega_0}}
\newcommand{\Vol}[1][n]{{Vol\left(\S^{#1}\right)}}
\DeclareMathOperator{\Lip}{Lip}

\author{Jun Kitagawa\footnote{Material from all sections excluding~\autoref{section: Wass2grad} represents a portion of the Ph.D. thesis of this author submitted at Princeton University.} and Micah Warren\thanks{ M.W. is supported in part by NSF
Grant DMS-0901644.}}
\title{Regularity for the optimal transportation problem with Euclidean distance squared cost on the embedded sphere}

\begin{document}

\maketitle
\begin{abstract}
We give sufficient conditions on initial and target measures supported on the sphere $\S^n$ to ensure the solution to the optimal transport problem with the cost $\cost{x}{y}$ is a diffeomorphism.
\end{abstract}

\section{Introduction}
In this paper, we will show the following two theorems.
\begin{thm}\label{thm: main}
Suppose that we have two probability measures $\mu:=e^{f}dVol$ and $\nu:=e^{g}dVol$ on $\S^n\subseteq \R^{n+1}$, with $f$ and $g$ smooth functions. If
\begin{equation}\label{density condition}
\funcnorm{\Sdiff{f}}+\funcnorm{\Sdiff{g}}<\frac{(n-1)\thmconst}{\pi}
\end{equation}
where $\thmconst$ satisfies
\begin{equation}\label{thmconst}
\thmconst e^{\thmconst}=2,
\end{equation}
then the optimal pairing of $\mu$ and $\nu$ under the optimal transport problem with cost given by the Euclidean distance squared on $\R^{n+1}$ is a diffeomorphism.
\end{thm}

\begin{thm}\label{thm: main 2}
Suppose that we have two smooth probability measures $\mu:=\rho dVol$ and
$\nu:= \bar{\rho}dVol$ on $\S^{n}\subseteq\R^{n+1}$ such that
\[
W_{2}^{2} (\mu,\nu)\leq\max\{\min_{x\in \S^{n}}\rho(x),\min_{x\in \S^{n}}%
\bar{\rho}(x)\} \Delta_1(n) 
\]
where
\begin{equation}\label{thmconst2}
\Delta_1(n) := \frac{\Vol[n-2]}{n(n+1)(n+2)}\left(\frac{2}{\pi}\right)^{n+2}\int_{0}^{\pi/2}\cos^{n+2}\phi\sin^{n-2}\phi d\phi,
\end{equation} 
and $W_{2}^{2}$ is the Wasserstein distance, computed with respect to cost given by Euclidean distance squared on $\R^{n+1}$. Then the optimal pairing of
$\mu$ and $\nu$ under the optimal transport problem is a diffeomorphism.
\end{thm}

Additionally, as a Corollary of Theorem~\ref{thm: main 2} we obtain
\begin{cor}\label{cor: main cor}
Suppose that we have two smooth probability measures $\mu:=\rho dVol$ and
$\nu:= \bar{\rho}dVol$ on $S^{n}\subseteq\R^{n+1}$ such that

\[
\lVert \rho-\bar{\rho}\rVert_{L^\infty(\S^n)}\leq\max\{\min_{x\in \S^{n}}\rho(x),\min_{x\in \S^{n}}%
\bar{\rho}(x)\}    \Delta_2(n)  
\]

where

\begin{equation}\label{thmconst3}
\Delta_2(n) := \frac{\pi\Vol[n-2]\Vol}{n(n+1)(n+2)}\left(  \frac{2}{\pi}\right)
^{n+2}\int_{0}^{\pi/2}\cos^{n+2}\phi\sin^{n-2}\phi d\phi.
\end{equation}
Then the optimal pairing of
$\mu$ and $\nu$ under the optimal transport problem with cost given by the Euclidean distance squared on $\R^{n+1}$ is a diffeomorphism.
\end{cor}
For the case when the cost is given by the geodesic distance squared on $\S^n$, regularity for general smooth positive densities was shown by Loeper in \cite{Loe09}. The cost given by the Euclidean distance squared was first investigated by Gangbo and McCann in \cite{GM00}, where the authors show examples of measures given by smooth densities where the optimal pairing is not given by a map. Thus there is a need for some condition on the two measures.

The idea behind both proofs is essentially the same: Follow a continuity method, considering solutions to the elliptic optimal transport equation (defined in~\autoref{section: setup} below).  If we can show that the image $\Tplus(x)$ (also defined in~\autoref{section: setup} below) of a point $x\in \S^n$ remains close enough to $x$, then derivative estimates follow from arguments of Ma, Trudinger, and Wang in~\cite{MTW05}.  In Theorem~\ref{thm: main} we obtain this closeness by first showing a gradient estimate of the solution $u$ to the optimal transportation equation, while in Theorem~\ref{thm: main 2} we derive the closeness directly from a bound on the $W_2$--Wasserstein distance associated to the Euclidean distance squared on $\R^{n+1}$. Additionally, we are able to use a simple estimate to obtain a bound on the $W_2$--Wasserstein distance from a bound on the $L^\infty$ difference of $\rho$ and $\bar{\rho}$, which allows us to immediately deduce Corollary~\ref{cor: main cor} from Theorem~\ref{thm: main 2}.

Finally, we present a short outline of the remainder of the paper. In \autoref{section: setup} we give the setup of the problem. Since our cost does not satisfy the twist condition, we take one of the branches of the cost exponential function, then use that to define the elliptic equation~\eqref{eqn2}. In \autoref{section: MTW calculation}, we calculate the Ma-Trudinger-Wang tensor in a specific coordinate system. In \autoref{section: gradient estimate}, we prove a gradient estimate for solutions of~\eqref{eqn2} under appropriate conditions on $f$ and $g$.  In \autoref{section: Wass2grad} we show an estimate under the conditions on $W^2_2(\rho, \bar {\rho})$ given in the second theorem, and also prove a short lemma which allows us to obtain the estimate under conditions on $\lVert \rho-\bar{\rho}\rVert_{L^\infty(\S^n)}$.  In \autoref{section: nonsplitting} we show that the contact set for such solutions can only consist of one point, under the appropriate gradient bound. Finally, in \autoref{section: proof of thm} we use the continuity method to prove our main theorem.

\textbf{Acknowledgement: } The authors would like to thank Alessio Figalli for a conversation which lead to a better argument in \autoref{section: nonsplitting}.

\textbf{Notation:}  We provide here a reference table for the notation used in this paper.

\begin{tabular}{lll}
Notation & Definition & Location\\
\hline\\
$\thmconst$& Constant satisfying $\thmconst e^\thmconst =2$&~\eqref{thmconst}\\
$ \Delta_1(n)$ &$\tfrac{\Vol[n-2]}{n(n+1)(n+2)}\left(\tfrac{2}{\pi}\right)^{n+2}   \int_{0}^{\pi/2}\cos^{n+2}\phi\sin^{n-2}\phi d\phi$&~\eqref{thmconst2}\\
$\Delta_2(n)$&$\tfrac{\pi\Vol[n-2]\Vol}{n(n+1)(n+2)}\left(  \tfrac{2}{\pi}\right)^{n+2}\int_{0}^{\pi/2}\cos^{n+2}\phi\sin^{n-2}\phi d\phi$&~\eqref{thmconst3}\\
$\norm{\cdot}$& Euclidean norm on $\R^{n+1}$\\
$\inner{\cdot}{\cdot}$& Euclidean inner product on $\R^{n+1}$\\
$\g(\cdot, \cdot)$& Canonical metric on $\S^n$\\
$\gnorm{\cdot}$& Length of vectors and covectors on $\S^n$\\
$\funcnorm{\Sdiff{u}}$&$\sup_{x\in \S^n}\gnorm{\Sdiff{u(x)}}$\\
$c_{ij ,kl}(x, y)$ etc. &$\pdiff{i}\pdiff{j}\pdiff[y]{k}\pdiff[y]{l} c(x, y)$ etc.\\
$c^{i, j}$ & inverse matrix of $c_{i, j}$\\
$w_{ij}(x)$&$u_{ij}(x, t)+c_{ij}(x,\Tplus[u](x))$&~\eqref{notation: w_ij}\\
$\Yplus(x, p)$&Inverse of a branch of the map $y\mapsto\Sdiff[x]{c(x, y)}=-p$&~\eqref{branch of inverse}\\
$\Tplus(x)$, $\Tplus[u](x)$&$\Yplus(x, \Sdiff{u(x)})$&before ~\eqref{eqn2}\\
$\Yminus(x, p)$&Inverse of a branch of the map $y\mapsto\Sdiff[x]{c(x, y)}=-p$&~\eqref{yminus}\\
$\Tminus[u](x)$&$\Yminus(x, \Sdiff{u(x)})$&~\eqref{tminus}
\end{tabular}
\\
\section{Set up of problem}\label{section: setup}
\subsection{Monge and Kantorovich problems}

For probability measure spaces $(\calX, \mu)$ and $(\calY, \nu)$, let $\marginals$ be the set of probability measures $\gamma$ on $\calX\times\calY$ such that
\begin{align*}
\gamma(E\times\calY)&=\mu(E)\\
\gamma(\calX\times \tilde{E})&=\nu(\tilde{E})
\end{align*}
for all $E\subset\calX$ and $\tilde{E}\subset\calY$ measurable. The cost is a measurable function 
\begin{equation*}
c:\calX\times\calY\to\R^+.
\end{equation*}
\begin{defin}\label{defin: Kant}
A probability measure $\gamma_0\in\marginals$ is a Kantorovich solution to the optimal transportation problem between $\mu$ and $\nu$ with cost $c$ if
\begin{equation*}
\int_{\calX\times\calY}c(x, y)d\gamma_0(x, y)=\inf_{\gamma\in\marginals}\int_{\calX\times\calY}c(x, y)d\gamma(x, y).
\end{equation*}
\end{defin}

\begin{defin}\label{defin: MongeProb}
A measurable map $T:\calX\to\calY$ is a Monge solution to the optimal transportation problem between $\mu$ and $\nu$ with cost $c$ if
\begin{equation*}
\int_{\calX}c(x, T(x))d\mu(x)=\inf_{S_\#\mu=\nu}\int_{\calX}c(x, S(x))d\mu(x).
\end{equation*}
\end{defin}

\begin{defin}\label{defin: c- convex function}
A real valued function $u$ defined on $\calX$ is $c$-convex if for each $x_0\in\calX$, there exists a $\lambda_0\in\R$ and $y_0\in\calY$ such that
\begin{align*}
u(x_0)&=-c(x_0, y_0)+\lambda_0\\
u(x)&\geq -c(x, y_0)+\lambda_0,\ \forall x\in\calX.
\end{align*}
If the second inequality is strict for $x\neq x_0$, we say the function is strictly $c$-convex.
We call such a function $-c(\cdot, y_0)+\lambda_0$ a $c$-support function to $u$ at $x_0$.
\end{defin}

\begin{defin}\label{defin: c transform}
For a $c$-convex function $u$, we define its $c$-transform $u^c$ by
\begin{equation*}
u^c(y):=\sup_{x\in\calX}{(-c(x, y)-u(x))}.
\end{equation*}
\end{defin}

\begin{rmk}\label{c-convexity inequality}
If $u$ is $c$-convex, then at any fixed $x_0$ where $u$ and $c$ are differentiable, we can see that for some $y_0\in\calY$,
\begin{align}\label{c-support relation}
u_i(x_0)&=-c_i(x_0, y_0)\notag\\
u_{ij}(x_0)&\geq -c_{ij}(x_0, y_0)
\end{align}
where the second inequality is in the sense of matrices. 
\end{rmk}

\subsubsection{Twisted case}
Now suppose that for each $x_0$ the map $\Sdiff[x]{c(x_0,\cdot)}:\calY \to \cotansp{x_0}{\calX}$ is defined and injective.  In this case we can implicitly define $T(x)$ from a $c$-convex function $u$ as
\begin{equation*}
-u_i(x)=c_i(x, T(x)).  
\end{equation*} Differentiating this and taking the determinant, one obtains

\begin{equation*}
\det{\left(u_{ij}(x)+c_{ij}(x,T(x))\right)} =\lvert\det{c_{i,j}(x, T(x))}\rvert \det{\left(\Jact{i}{j}(x)\right)} 
\end{equation*}
and from this we can write down the elliptic optimal transport equation:
\begin{equation}\label{eqn}
\det{\left(u_{ij}(x)+c_{ij}(x,T(x))\right)}=\lvert\det{c_{i,j}(x, T(x))}\rvert\frac{\mu(x)}{\nu(T(x))}, 
\end{equation}
here $\mu(x)$ and $\nu(y)$ are densities with respect to $dx^1\cdots dx^n$ and $dy^1\cdots dy^n$ for chosen coordinate systems (assuming the measures are absolutely continuous with respect to a coordinate volume form).  This equation is (degenerate) elliptic for a $c$-convex solution $u$ (see Remark ~\ref{c-convexity inequality}), and is invariant under a change of coordinates.
It is a standard result that $c$-convex solutions of the above equation determine uniquely the solution to the optimal transportation problem. 

\subsection{$\S^n$ with Euclidean cost} 

For the remainder of the paper, we specialize to the case
\begin{equation*}
c(x, y):=\cost{x}{y},
\end{equation*}where $\vert\cdot\rvert$ is the Euclidean distance on $\R^{n+1}$. 

In the twisted case, it is clear that each $c$-convex function determines a single-valued map $T(x)$.  This is no longer the case for this cost function on the sphere:  For a fixed $x_0\in\S^n$, the map $\Sdiff[x]{c(x_0,\cdot)}:\S^n \to \cotansp{x_0}{\S^n}$ is not injective. In fact, if $p\in\cotansp{x_0}{\S^n}$, $\gnorm{p}<1$, there are exactly two points $y=y_1$ and $y=y_2$ such that $\Sdiff[x]{c(x_0, y)}=-p$.  With this in mind, we make the following definition.

\begin{defin}\label{branch of inverse}
We define the map $\Yplus(x, p)$ implicitly by
\begin{align*}
\Sdiff[x]{c(x, \Yplus(x, p))}&=-p\\
\inner{x}{\Yplus(x, p)}>0
\end{align*}
for $p\in\cotansp{x}{\S^n}$, $\gnorm{p}<1$.
\end{defin}
Now for $u$ $c$-convex with $\funcnorm{Du}<1,$ we use $\Yplus$ in place of the implicit definition of $T(x)$  in~\eqref{eqn}  to define an elliptic equation
\begin{equation}\label{eqn2}
\det{w_{ij}(x)}=\lvert\det{c_{i,j}(x, \Tplus(x))}\rvert\frac{\mu(x)}{\nu(\Tplus(x)))}, 
\end{equation} 
where 
\begin{equation}\label{notation: w_ij}
w_{ij}(x):=u_{ij}(x)+c_{ij}(x,\Tplus(x))
\end{equation}
and
\begin{equation*}
\Tplus(x) :=\Yplus(x, \Sdiff{u(x)})
\end{equation*}
which is well-defined as long as $\funcnorm{Du}<1$. 
Eventually, we will show that under appropriate conditions on the densities, solutions of this equation indeed determine single-valued solutions to the optimal transport problem.  

\begin{rmk}\label{rmk: implicit definition of T}
If $u$ is $c$-convex and differentiable at $x_0$  with $\gnorm{Du(x_0)}<1$, and $(x_0, y_0)$ is a pair of points satisfying~\eqref{c-support relation}, it is easy to see that either $\inner{x_0}{y_0}>0$ or there exists another point $y_1$ also satisfying~\eqref{c-support relation} with $\inner{x_0}{y_1}>0$. In the latter case $c_{ij}(x_0, y_1)>c_{ij}(x_0, y_0)$. Hence, we see that: 
\begin{align*}
&\text{the pair }(x_0, \Tplus(x_0))\text{ satisfies}~\eqref{c-support relation},\\
&w_{ij}(x_0)\text{ is positive semidefinite}.
\end{align*}
\end{rmk}
\section{Calculation of MTW tensor}\label{section: MTW calculation}
In this section, we utilize a coordinate system specialized to this problem on the sphere to calculate various quantities involving $c$, most notably the MTW tensor of \cite{MTW05}.

We will define a coordinate system centered around a point $x_0\in\S^n$ as follows. First, rotate $x_0$ so it is given by $e_{n+1}$, then take coordinates for the upper hemisphere by representing it as the graph $(x, \factor{x})$ over $B_1(0)\subset \R^n$, where
\begin{equation*}\label{definition of beta}
\factor{x}:=\sqrt{1-\norm{x}^2}.
\end{equation*}
Note that we leave ourselves the freedom to further rotate $\S^n$ as long as the $e_{n+1}$ direction remains unchanged. 
\begin{defin}\cite[Section 2]{MTW05}
Given $V$, $W\in \tansp{x}{\S^n}$ and $\eta$, $\zeta\in \cotansp{x}{\S^n}$, define
\begin{equation*}\label{MTW}
\MTW{i}{j}{k}{l}(x, y)V^iW^j\eta_k\zeta_l:=\MTWcoord{i}{j}{k}{l}(x, y)V^iW^j\eta_k\zeta_l.
\end{equation*}
We will say that a cost $c$ 
has the property~\eqref{A3s} at $(x, y)\in\StimesS$ with constant $\delta_0>0$ if
\begin{equation}\label{A3s}\tag{A3s}
\MTW{i}{j}{k}{l}(x, y)V^iV^j\eta_k\eta_l\geq \delta_0 \gnorm{V}^2\gnorm{\eta}^2.
\end{equation}
\end{defin}

The following is widely known, but we carry out the calculations in our coordinate system here for later reference.
\begin{prop}\label{prop: c is A3s}
The cost $\cost{x}{y}$ satisfies~\eqref{A3s} with a uniform constant $\delta_0=1$ for all $(x, y)\in\StimesS$ such that $\dist{x}{y}>0$.
\end{prop}
\begin{proof}
We will utilize the coordinate system indicated above, and calculate various derivatives of $c$. First,
\begin{align*}
c(x, y)&=\cost{(x, \factor{x})}{(y, \factor{y})}\\
&=\frac{\norm{x-y}^2+(\factor{x}-\factor{y})^2}{2}\\
&=\frac{\norm{x}^2+\norm{y}^2-2\inner{x}{y}+(1-\norm{x}^2+1-\norm{y}^2-2\factor{x}\factor{y})}{2}\\
&=1-\inner{x}{y}-\factor{x}\factor{y}.
\end{align*}
Thus we calculate, at generic $x$ and $y$,
\begin{align}\label{derivatives of c}
c_i&=-y_i-\factorsub{x}{i}\factor{y}\notag\\
c_{ij}&=-\factorsub{x}{ij}\factor{y}\notag\\
c_{ij, k}&=-\factorsub{x}{ij}\factorsub{y}{k}\notag\\
c_{i,k}&=-\delta_{ik}-\factorsub{x}{i}\factorsub{y}{k}\notag\\
c_{i, kl}&=-\factorsub{x}{i}\factorsub{y}{kl}\notag\\
c_{ij, kl}&=-\factorsub{x}{ij}\factorsub{y}{kl}
\end{align}
and
\begin{align*}
\factorsub{x}{i}&=-\frac{x_i}{\factor{x}}\notag\\
\factorsub{x}{ij}&=-\frac{\delta_{ij}}{\factor{x}}-\frac{x_ix_j}{\factor{x}^3}.
\end{align*}
Thus we find at $x=0$,
\begin{align}\label{derivatives of c at zero}
c_i&=-y_i\notag\\
c_{ij}&=\delta_{ij}\factor{y}\notag\\
c_{ij, k}&=-\delta_{ij}\frac{y_k}{\factor{y}}\notag\\
c_{i,k}&=-\delta_{ik}\notag\\
c_{i, kl}&=0\notag\\
c_{ij, kl}&=-\delta_{ij}\left(\frac{\delta_{kl}}{\factor{y}}+\frac{y_ky_l}{\factor{y}^3}\right).
\end{align}
Now we can calculate for $V\in \tansp{0}{\S^n}$, $\eta\in \cotansp{0}{\S^n}$,
\begin{align*}
\MTW{i}{j}{k}{l}V^iV^j\eta_k\eta_l&=\MTWcoord{i}{j}{k}{l}V^iV^j\eta_k\eta_l\notag\\
&=\delta_{ij}\left(\frac{\delta_{pq}}{\factor{y}}+\frac{y_py_q}{\factor{y}^3}\right)(-\delta^{pk})(-\delta^{ql})V^iV^j\eta_k\eta_l\notag\\
&=\frac{\norm{V}^2}{\factor{y}}\left(\delta_{kl}+\frac{y_ky_l}{\factor{y}^2}\right)\eta_k\eta_l\notag\\
&\geq\norm{V}^2\norm{\eta}^2\\
&=\gnorm{V}^2\gnorm{\eta}^2
\end{align*}
since 
\begin{align*}
\sum_{k,l}y_k\eta_ky_l\eta_l&=\inner{y}{\eta}^2\geq 0,\\
\factor{y}&\leq 1
\end{align*}
for any $y$ and $\eta$. The last equality is seen from calculation of the metric in our coordinates, shown below.
\end{proof}

We also make a few calculations for later use. In the above coordinates, 
\begin{align*}
\g_{ij}
&=\delta_{ij}+\factorsub{x}{i}\factorsub{x}{j}\\
(\g_{ij})_k&=\factorsub{x}{ik}\factorsub{x}{j}+\factorsub{x}{i}\factorsub{x}{jk}\\
(\g_{ij})_{kl}&=\factorsub{x}{ikl}\factorsub{x}{j}+\factorsub{x}{ik}\factorsub{x}{jl}+\factorsub{x}{il}\factorsub{x}{jk}+\factorsub{x}{i}\factorsub{x}{jkl}
\end{align*}
and at $x=0$ we have
\begin{align*}
\g_{ij}&=\delta_{ij}\notag\\
(\g_{ij})_k&=0\notag\\
(\g_{ij})_{kl}&=\delta_{ik}\delta_{jl}+\delta_{il}\delta_{jk}.
\end{align*}
Hence at $x=0$,
\begin{align}\label{dervatives of the metric}
\g^{ij}&=\delta_{ij}\notag\\
(\g^{ij})_k&=0\notag\\
(\g^{ij})_{kl}&=-\delta_{ik}\delta_{jl}-\delta_{il}\delta_{jk}.
\end{align}
\begin{rmk}\label{tplus in coordinates}
Suppose we have a $C^1$, $c$-convex function $u$. Then, when written in a coordinate system chosen as above centered at $x_0$, using the implicit relation in Definition~\ref{branch of inverse} we notice that 
\begin{equation*}
\left(\Tplus(x_0)\right)^i=u_i(x_0).
\end{equation*}
\end{rmk}

\section{Gradient estimate : Theorem 1.1}\label{section: gradient estimate}
In this section, we will prove an {\em a priori} gradient estimate for a solution $u$ of our elliptic equation~\eqref{eqn2}. A gradient bound of the form $\funcnorm{\Sdiff{u}}<1-\epsilon$ will ensure that Proposition~\ref{prop: c is A3s} is applicable at $(x, \Tplus(x))$, hence we can use the MTW theory to to obtain {\em a priori} second derivative estimates for $u$. The method we use is similar to that of Delano{\"e} and Loeper in \cite{DL06}, where the cost is the geodesic distance squared. In~\eqref{eqn2} we will take
\begin{align*}
\mu(x)&:=\rhoone[(x)]\\
\nu(y)&:=\rhobar[(y)].
\end{align*}
\begin{thm}\label{thm: stay away property}
Let $n\geq 2$. Suppose $u$ is a $C^2$ solution to equation~\eqref{eqn2} such that 
\begin{align*}
\funcnorm{\Sdiff{u}}&<1.
\end{align*}
Then, 
\begin{equation*}
\funcnorm{\Sdiff{u}}\leq\left(\frac{1}{n-1}\right)\left(\frac{\Lambdaf}{\lambdag}\right)^{\frac{1}{n-1}}\calD{f}{g}.
\end{equation*}
\end{thm}

\begin{proof}
Define
\begin{equation*}
\phi(x):=\frac{\gnorm{\Sdiff{u(x)}}^2}{2}
\end{equation*}
where $u$ is a solution to the equation~\eqref{eqn2}. Now we let $x_0$ be the point where $\phi$ achieves its maximum on $\S^n$, and take the coordinate system defined in \autoref{section: MTW calculation} centered at this point. We take $dx^1$ in the direction of $\Sdiff{u}$, and rotate the remaining $n-1$ directions so that $u_{ij}$ for $1<i, j\leq n$ is diagonal at $x_0=$. Note that at $x=0$,
\begin{align*}
0&=\phi_i\notag\\
&=\frac{(\g^{pq})_iu_pu_q+2\g^{pq}u_pu_{qi}}{2}\notag\\
&=u_1u_{1i}
\end{align*}
and
\begin{align*}
\phi_{ij}&=\frac{(\g^{pq})_{ij}u_pu_q+2(\g^{pq})_iu_pu_{qj}+2(\g^{pq})_ju_pu_{qi}+2\g^{pq}u_pu_{qij}+2\g^{pq}u_{pj}u_{qi}}{2}\notag\\
&=u_1^2\frac{(\g^{11})_{ij}}{2}+u_1u_{1ij}+\sum_pu_{pi}u_{pj}\notag\\
&=-u_1^2\delta_{i1}\delta_{j1}+u_1u_{1ij}+\sum_pu_{pi}u_{pj}.
\end{align*}
Here we have used~\eqref{dervatives of the metric}.
Now if $u_1(0)=0$, that implies that $\funcnorm{\Sdiff{u}}=0$ and $u$ is constant. Thus we may assume $u_1(0)\neq 0$ and hence 
\begin{equation*}
u_{1i}(0)=0
\end{equation*}
for all $1\leq i \leq n$. In particular, the whole matrix $u_{ij}$ is diagonal at $0$.

Consider the operator
\begin{equation}\label{linearization}
Lv:=w^{ij}v_{ij}
\end{equation}
which is the second order part of the linearization of the natural logarithm of~\eqref{eqn2}.
Taking $v=\phi$ and $x=0$, and applying the maximum principle we find that
\begin{align}\label{Lphi}
0&\geq L\phi(0)\notag\\
&=w^{ij}\phi_{ij}\notag\\
&=w^{ij}(-u_1^2\delta_{i1}\delta_{j1}+u_1u_{1ij}+\sum_pu_{pi}u_{pj})\notag\\
&=-u_1^2w^{11}+u_1w^{ij}u_{1ij}+w^{ij}\sum_pu_{pi}u_{pj}\notag\\
&=-\frac{u_1^2}{\factor{\Sdiff{u}}}+\sum_\alpha w^{\alpha\alpha}(u_1u_{1\alpha\alpha}+u_{\alpha\alpha}^2).
\end{align}
Here we have used~\eqref{derivatives of c at zero} and the fact that $u_{ij}$ is diagonal at $0$.

By differentiating the implicit relation
\begin{equation*}
u_i(x)+c_i(x, \Tplus(x))=0
\end{equation*}
we find that 
\begin{align*}
u_{ij}+c_{ij}&=-c_{i, k}(\Tplus)^k_j.
\end{align*}
Thus from~\eqref{derivatives of c at zero} and Remark~\ref{tplus in coordinates},
\begin{align*}
(\Tplus)^i_j(0)&=u_{ij}(0)+\delta_{ij}\factor{\Tplus(0)}\\
&=u_{ij}(0)+\delta_{ij}\factor{\Sdiff{u(0)}}.
\end{align*}
In particular, $(\Tplus)^i_j(0)$ is diagonal. Note additionally, from this we find that
\begin{equation}\label{w inverse}
w^{ij}(0)=\frac{1}{u_{ij}(0)+\delta_{ij}\factor{\Sdiff{u(0)}}}.
\end{equation}
Also, from~\eqref{derivatives of c} we calculate
\begin{equation*}
c_{ij1}(0, y)=-\factorsub{0}{ij1}\factor{y}=0.
\end{equation*}
By the calculations at the end of \autoref{section: MTW calculation}, we have
\begin{align*}
\rhoone[(x)]&=\frac{e^{f(x)}}{\factor{x}}\\
\rhobar[{(\Tplus(x))}]&=\frac{e^{g(\Tplus(x))}}{\factor{\Tplus(x)}}.
\end{align*}
We now differentiate equation~\eqref{eqn2} (after taking logarithms again) in the $x_1$ direction to obtain, at $x=0$,
\begin{align*}
0&=w^{ij}(u_{ij1}+c_{ij1}+c_{ij, k}(\Tplus)^k_1)-c^{i, j}(c_{i1, j}+c_{i, jk}(\Tplus)^k_1)\\
&\qquad\qquad\qquad-\Dx[1](f-\log{\factornon})+\Dx[1](g\circ \Tplus-\log{\factor{\Tplus}})\\
&=\sum_\alpha [w^{\alpha\alpha}(u_{\alpha\alpha 1}+c_{\alpha\alpha, 1}(\Tplus)^1_1)+c_{\alpha 1, \alpha}]\\
&\qquad\qquad\qquad-f_1+\frac{\factorsubnon{1}}{\factornon{}}+(g_k\circ\Tplus)(\Tplus)^k_1-\frac{\factorsub{\Tplus}{k}(\Tplus)^k_1}{\factor{\Tplus}}\\
&=\sum_\alpha [w^{\alpha\alpha}(u_{\alpha\alpha 1}+c_{\alpha\alpha, 1}(\Tplus)^1_1)+c_{\alpha 1, \alpha}]\\
&\qquad\qquad\qquad-f_1+\frac{\factorsubnon{1}}{\factornon{}}+(g_1\circ\Tplus)(\Tplus)^1_1-\frac{\factorsub{\Tplus}{1}(\Tplus)^1_1}{\factor{\Tplus}}\\
&=\sum_\alpha [w^{\alpha\alpha}(u_{\alpha\alpha 1}-u_1)]-\frac{u_1}{\factor{\Sdiff{u}}}\\
&\qquad\qquad\qquad-f_1+g_1(\Sdiff{u})\factor{\Sdiff{u}}+\frac{u_1}{\factor{\Sdiff{u}}}.
\end{align*}
Substituting into~\eqref{Lphi}, we obtain
\begin{align*}
0&\geq-\frac{u_1^2}{\factor{\Sdiff{u}}}+\sum_\alpha w^{\alpha\alpha}u_{\alpha\alpha}^2+u_1\left(\sum_\alpha w^{\alpha\alpha}u_1+f_1-g_1(\Sdiff{u})\factor{\Sdiff{u}}\right)\\
&=-\frac{u_1^2}{\factor{\Sdiff{u}}}+\sum_\alpha w^{\alpha\alpha}u_{\alpha\alpha}^2+u_1^2\sum_\alpha w^{\alpha\alpha}+u_1f_1-u_1g_1(\Sdiff{u})\factor{\Sdiff{u}}\\
&\geq u_1^2\sum_\alpha w^{\alpha\alpha}+u_1f_1-u_1g_1(\Sdiff{u})\factor{\Sdiff{u}}-\frac{u_1^2}{\factor{\Sdiff{u}}}\\
&=u_1^2\sum_{\alpha>1} w^{\alpha\alpha}+u_1f_1-u_1g_1(\Sdiff{u})\factor{\Sdiff{u}}.
\end{align*}
Here we have used that $w^{\alpha\alpha}(0)\geq 0$ and $w^{11}(0)=\frac{1}{\factor{\Sdiff{u(0)}}}$ by~\eqref{w inverse}. Thus we find that
\begin{align*}
\calD{f}{g}
&\geq\left\vert f_1\right\vert+\left\vert g_1(\Sdiff{u})\factor{\Sdiff{u}}\right\vert\\
&\geq u_1\sum_{\alpha>1} w^{\alpha\alpha}\\
&\geq u_1(n-1)(\prod_{\alpha>1}{w^{\alpha\alpha}})^{\frac{1}{n-1}}\\
&=(n-1)u_1\left(\frac{e^{g(\Sdiff{u})}}{e^{f}}\right)^{\frac{1}{n-1}}\\
&\geq(n-1)u_1\left(\frac{\lambdag}{\Lambdaf}\right)^{\frac{1}{n-1}}\\
&=(n-1)\left(\frac{\lambdag}{\Lambdaf}\right)^{\frac{1}{n-1}}\funcnorm{\Sdiff{u}}.
\end{align*}
Hence
\begin{equation*}
\funcnorm{\Sdiff{u}}\leq\left(\frac{1}{n-1}\right)\left(\frac{\Lambdaf}{\lambdag}\right)^{\frac{1}{n-1}}\calD{f}{g}.
\end{equation*}
\end{proof}

\section{Gradient estimate : Theorem 1.2 }\label{section: Wass2grad}

We will prove the following interior estimate on Euclidean space itself and
then show that it adapts easily to hold on embedded spheres as well.

\begin{thm}
Suppose that $\lvert T(x_{0})-x_{0}\rvert =a$ , where $T$ is an
optimal transport map from $\rho(x)dx$ to $\bar{\rho}(y)dy$ in $\R^{n}$, with $\rho$, $\bar{\rho}$ supported in domains $\Omega,\bar{\Omega}$ and $d(x_0,\partial\Omega)\geq a$. Then
\[
a\leq\left(  \frac{1}{\min_{x\in\Omega}\rho(x)}\frac{n(n+1)(n+2)}{\Vol[n-2] }\frac{1}{\int_{0}^{\pi/2}\cos^{n+2}\phi\sin^{n-2}\phi
d\phi}W_{2}^{2}(\rho,\bar{\rho})\right)  ^{1/(n+2)}%
\]
where $W_{2}^{2}(\rho,\bar{\rho})$ is the Wasserstein distance between $\rho$
and $\bar{\rho}.$

\end{thm}

\begin{proof}
Rotate coordinates on $\R^{n}$ so that $x_{0}$ is at the origin and $T(x_{0})$ is along the $e_{1}$
axis. Define the set
\[
K_{a}=\left\{  x\in\R^{n}\mid\lvert x\rvert\leq a\cos\phi\right\}
\]
where $\phi$ is the angle $x$ makes with the $e_{1}$ axis. This set is a sphere of radius $a/2$ centered at the point $(a/2,0,...,0)$.  Now take a point
point $x_{1}\in K_{a}$.   The monotonicity condition for optimal transport (cf. \cite[Def 5.1]{Vil09}) says that 
\begin{equation}
 \inner{x_0}{T(x_0)} +\inner{x_1}{T(x_1)} \geq  \inner{x_1}{T(x_0)}  +   \inner{x_0}{T(x_1)}  \label{mcityEuc} 
\end{equation} in particular 
$$ \inner{x_1}{T(x_1)} \geq \inner{x_1}{T(x_0)}= a \cos \phi  . $$  It follows that  $T(x_{1})$ must
be in a half space and that
\[
\lvert x_{1}-T(x_{1})\rvert\geq a\cos\phi-\lvert x_{1}\rvert .
\]

Thus we can integrate
\begin{align*}
W_{2}^{2}(\rho,\bar{\rho}) &  =\int\frac{\lvert x-T(x)\rvert^{2}}{2}\rho(x)dx\geq
\int_{K_{a}}\frac{\rvert x-T(x)\lvert^{2}}{2}\rho(x)dx\\
&  \geq\int_{0}^{\pi/2}\int_{S^{n-2}}\int_{0}^{a\cos\phi}\rho(x)\frac{\left(
a\cos\phi-r\right)  ^{2}}{2}r^{n-1}drd\sigma\sin^{n-2}\phi d\phi\\
&  \geq\min_{x\in\Omega}\rho(x)\frac{\Vol[n-2]}{2}%
\int_{0}^{\pi/2}\int_{0}^{a\cos\phi}\left( a\cos\phi-r\right)  ^{2}%
r^{n-1}drd\sigma\sin^{n-2}\phi d\phi\\
&  =\min_{x\in\Omega}\rho(x)\Vol[n-2] \frac{a^{n+2}%
}{n(n+1)(n+2)}\int_{0}^{\pi/2}\cos^{n+2}\phi\sin^{n-2}\phi d\phi
\end{align*}
and the conclusion follows.
\end{proof}

\begin{cor}\label{cor: wasserstein gradient bound on sphere}
The same estimate holds on the sphere $\S^{n}\subset%
\R^{n+1}$ with cost given by Euclidean distance squared on $\R^{n+1}$.  In particular 
\begin{equation*}
\funcnorm{\Sdiff{u}}\leq \left(  \frac{1}{\min_{x\in\Omega}\rho(x)}\frac{n(n+1)(n+2)}{\Vol[n-2] }\frac{1}{\int_{0}^{\pi/2}\cos^{n+2}\phi\sin^{n-2}\phi
d\phi}W_{2}^{2}(\rho,\bar{\rho})\right)  ^{1/(n+2)}.
\end{equation*}
\end{cor}

\begin{proof}
Represent a hemisphere as a graph over the tangent space at $x_{0}.$ Then
using these coordinates, repeat the above calculation. First note that the
distance of the projection in the chosen coordinates bounds from below the actual
distance. An easy computation shows that the monotonicity condition for this cost (in these coordinates) is stronger than (\ref{mcityEuc}.) 
It follows that our integration argument over $K_{a}$ is intact, 
noting that the volume element on the sphere is bounded below by $dx$
in these coordinates.  The derivative bound follows by noting that the norm of $Du$ is given by the length of tangential component of $T(x)-x$.  
\end{proof}

\begin{rmk}This method can be modified to uniformly convex domains, where the constants depend explicitly on the upper and lower curvature bounds. The set $K_{a}$ is less nice, but still explicit. One can repeat almost verbatim the same argument on a manifold with nonnegative curvature (with respect to distance squared cost), provided one has a lower bound on the volume element in exponential coordinates. 
\end{rmk}

We also show that a bound on the $L^\infty$ distance between $\rho$ and $\bar{\rho}$ implies a bound on the $W_2$--Wasserstein distance between them, which easily allows us to prove Corollary~\ref{cor: main cor} from Theorem~\ref{thm: main}.
\begin{lemma}\label{lemma: L^infty bounds wasserstein}
Given any two probability measures $\mu=\rho dVol$ and $\nu=\bar{\rho} dVol$ such that $\rho$, $\bar{\rho}\in L^\infty(\S^n)$, then
\begin{equation*}
W^2(\mu, \nu)\leq \pi\Vol\lVert \rho-\bar{\rho}\rVert_{L^\infty(\S^n)}.
\end{equation*}
\end{lemma}
\begin{proof}
Recall by~\cite[Chapter 5]{Vil09} that
\begin{equation}\label{eqn: dual equality}
W^2(\mu, \nu)=\sup_{(\phi, \psi)\in \mathcal{K}} -\int_{\S^n}\phi(x)\rho(x) dVol(x) -\int_{\S^n}\psi(y)\bar{\rho}(y)dVol(y)
\end{equation}
where
\begin{equation*}
\mathcal{K}:=\{(\phi, \psi)\in C(\S^n)\times C(\S^n)\mid -\phi(x)-\psi(y)\leq c(x, y)\}.
\end{equation*}
Now note that since $\mu$ and $\nu$ are both probability measures, we may add the restriction $\phi(e_n)=\psi(e_n)=0$ to the definition of the set $\mathcal{K}$ without changing the supremum in~\eqref{eqn: dual equality}. Also, by definition we see that
\begin{align*}
&\sup_{(\phi, \psi)\in \mathcal{K}} \left(-\int_{\S^n}\phi(x)\rho(x) dVol(x) -\int_{\S^n}\psi(y)\bar{\rho}(y)dVol(y)\right)\\
&\qquad\leq\sup_{\phi\in \{C(\S^n)\mid \phi(e_n)=0\}} \left(-\int_{\S^n}\phi(x)\rho(x) dVol(x) -\int_{\S^n}\phi^c(y)\bar{\rho}(y)dVol(y)\right)\\
&\qquad\leq\sup_{\phi\in \{C(\S^n)\mid \phi(e_n)=0\}} \left(-\int_{\S^n}\phi^{cc}(x)\rho(x) dVol(x) -\int_{\S^n}\phi^c(y)\bar{\rho}(y)dVol(y)\right)\\
&\qquad\leq\sup_{\phi\in \{C(\S^n)\mid \phi(e_n)=0\}} \left(-\int_{\S^n}\phi^{cc}(x)\rho(x) dVol(x) -\int_{\S^n}[-c(y, y)-\phi(y)]\bar{\rho}(y)dVol(y)\right)\\
&\qquad=\sup_{\phi\in \{C(\S^n)\mid \phi(e_n)=0\}} \left(-\int_{\S^n}\phi^{cc}(x)\rho(x) dVol(x) +\int_{\S^n}\phi(y)\bar{\rho}(y)dVol(y)\right)\\
&\qquad\leq\sup_{\phi\in \{C(\S^n)\mid \phi(e_n)=0\}} \left(\lVert\phi^{cc}\rVert_{L^\infty(\S^n)}\lVert \rho-\bar{\rho}\rVert_{L^\infty(\S^n)} \vol(\S^n)\right)\\
&\qquad\leq\sup_{\phi\in \{C(\S^n)\mid \phi(e_n)=0\}} \left(\lVert\phi^{cc}\rVert_{\Lip(\S^n)}\diam{(\S^n)}\lVert \rho-\bar{\rho}\rVert_{L^\infty(\S^n)} \vol(\S^n)\right)\\
&\qquad\leq\sup_{y\in\S^n}{\lVert \Sdiff[x]{c(\cdot, y)}\rVert}\diam{(\S^n)}\vol(\S^n)\lVert \rho-\bar{\rho}\rVert_{L^\infty(\S^n)} \\
&\qquad= \pi\Vol\lVert \rho-\bar{\rho}\rVert_{L^\infty(\S^n)}.
\end{align*}
Here we have used that since $\phi^{cc}$ is a $c$-convex function,
\begin{equation*}
\lVert \phi^{cc}\rVert_{\Lip(\S^n)}\leq \sup_{y\in\S^n}{\lVert \Sdiff[x]{c(\cdot, y)}}\rVert\leq 1.
\end{equation*}
\end{proof}

\section{Nonsplitting}\label{section: nonsplitting}
First we define the $c$-subdifferential of a $c$-convex function $u$ at a point $x$ by
\begin{defin}
\begin{equation*}
\csubdiff u(x):=\left\{y\in\S^n\vert -c(\cdot, y)+\lambda\text{ is a }c\text{-support function to }u \text{ at }x\text{ for some }\lambda\in\R\right\}.
\end{equation*}
\end{defin}
We also define
\begin{equation}\label{tminus}
\Tminus[u](x):=\Yminus(x, \Sdiff{u(x)})
\end{equation}
where $\Yminus$ is characterized by 
\begin{align}\label{yminus}
\Sdiff[x]{c(x, \Yminus(x, p))}&=-p\notag\\
\dist{x}{\Yminus(x, p)}&<0
\end{align}
for $p\in\cotansp{x}{\S^n}$, $\gnorm{p}<1$ (compare Definition~\ref{branch of inverse}). We add the subscript $u$ to emphasize the dependency on the potential function $u$.

Now we show a pointwise estimate on $\gnorm{\Sdiff{u(x)}}$ if $\csubdiff u(x)$ is more than one point for any $x$.

\begin{lemma}\label{lemma: doesn't split}
Suppose that $u$ is $c$-convex, $C^1$, and $\funcnorm{\Sdiff{u}}<\frac{2}{\pi}$. Then, 
\begin{equation*}
\csubdiff u(x_0)=\left\{\Tplus[u](x_0)\right\}.
\end{equation*}
\end{lemma}
\begin{proof}
Suppose that $\csubdiff u(x_0)\neq\left\{\Tplus[u](x_0)\right\}$. Then we must have $\Tminus[u](x_0)\in\csubdiff u(x_0)$. Writing $\yminus=\Tminus[u](x_0)$, this implies that for some $\lambda\in\R$, the function $-c(\cdot, \yminus)+\lambda$ is a $c$-support function to $u$ at $x_0$, and hence
\begin{align*}
u(\yminus)-u(x_0)&\geq -c(\yminus, \yminus)+\lambda-(-c(x_0, \yminus)+\lambda)\\
&=c(x_0, \yminus)\\
&=\cost{x_0}{\yminus}\\
&=1-\cos{(d_{\S^n}({x_0},{\yminus}))}.
\end{align*}
However, we also have
\begin{align*}
u(\yminus)-u(x_0)&\leq d_{\S^n}({x_0},{\yminus})\funcnorm{\Sdiff{u}}
\end{align*}
hence
\begin{equation*}
1\leq d_{\S^n}({x_0},{\yminus})\funcnorm{\Sdiff{u}}+\cos{(d_{\S^n}({x_0},{\yminus}))}.
\end{equation*}
However, by the definition of $\Tminus[u]$, we see that $\frac{\pi}{2}\leq d_{\S^n}({x_0},{\yminus})\leq \pi$. Thus 
by considering the real valued function $\funcnorm{\Sdiff{u}}t+\cos{t}$ on the interval $[\frac{\pi}{2}, \pi]$, we see that 
\begin{align*}
d_{\S^n}({x_0},{\yminus})\funcnorm{\Sdiff{u}}+\cos{(d_{\S^n}({x_0},{\yminus}))}&\leq \max{\left\{\frac{\pi}{2}\funcnorm{\Sdiff{u}}+\cos{\frac{\pi}{2}}, \pi\funcnorm{\Sdiff{u}}+\cos{\pi}\right\}}\\
&<1
\end{align*}
by the assumption on $\funcnorm{\Sdiff{u}}$, which is a contradiction.
\end{proof}

\section{Proof of Main Theorem}\label{section: proof of thm}
By combining the appropriate gradient estimate (either Theorem~\ref{thm: stay away property} or Corollary~\ref{cor: wasserstein gradient bound on sphere}) with Lemma~\ref{lemma: doesn't split}, we can use the continuity method to show the existence of a Monge solution to our problem, proving our two main theorems.

\begin{proof}[Proof of Theorem~\ref{thm: main}]

Assume $n\geq 2$, as the case $n=1$ is vacuous. 

We will apply the continuity method to the equations
\begin{equation}\label{cont eqn}
\det{(u_{ij}(x)+c_{ij}(x, \Tplus[u](x)))}=\lvert\det{c_{i,j}(x, \Tplus[u](x))}\rvert\frac{\rhoonet[(x)]}{\rhobart[{(\Tplus[u](x))}]}
\end{equation}
and
\begin{equation}\label{rev eqn}
\det{(v_{ij}(y)+c_{ij}(\Tplus[v](y), y))}=\lvert\det{c_{i,j}(\Tplus[v](y), y)}\rvert\frac{\rhobart[(y)]}{\rhoonet[{(\Tplus[v](y))}]}
\end{equation}
where, for $t\in[0, 1]$: 
\begin{align*}
f_t(x)&:=\log{\left(\frac{(1-t)}{\Vol}+tf(x)\right)}\\
g_t(y)&:=\log{\left(\frac{(1-t)}{\Vol}+tg(y)\right)}.
\end{align*}

Let 
\begin{equation*}
I:=\{t\in[0, 1]\ \vert\text{~\eqref{cont eqn} has a smooth solution }u\text{ such that }\funcnorm{\Sdiff{u}}<1\}.
\end{equation*}

Now a simple calculation shows that 
\begin{align*}
\funcnorm{\Sdiff{f_t}}&\leq\funcnorm{\Sdiff{f}}\\
\funcnorm{\Sdiff{g_t}}&\leq\funcnorm{\Sdiff{g}}.
\end{align*}
Since $e^{f_t}dVol_{\S^n}$ and $e^{g_t}dVol_{\S^n}$ are probability measures, $f_t$ and $g_t$ each equal $-\log{(\Vol)}$ at least once. Then, by assumption~\eqref{density condition} on $\calD{f}{g}$ we obtain that for some $0\leq \lambda \leq 1$, 
\begin{align*}
\min{g_t}&\geq-\log{(\Vol)}-\lambda\left((n-1)\thmconst\right)\\
\max{f_t}&\leq -\log{(\Vol)}+(1-\lambda)\left((n-1)\thmconst\right)\\
\frac{\Lambdaft}{\lambdagt}&\leq e^{\lambda(n-1)\thmconst+(1-\lambda)(n-1)\thmconst}=e^{(n-1)\thmconst}.
\end{align*}

Hence, for any $t\in I$ and a solution $\ut$ to~\eqref{cont eqn}, we may apply Theorem~\ref{thm: stay away property} to find
\begin{align}\label{Du est}
\funcnorm{\Sdiff{\ut}}&\leq\left(\frac{1}{n-1}\right)\left(\frac{\Lambdaft}{\lambdagt}\right)^{\frac{1}{n-1}}\calD{f_t}{g_t}\notag\\
&=\left(\frac{1}{n-1}\right)\left(e^{(n-1)\thmconst}\right)^{\frac{1}{n-1}}\calD{f}{g}\notag\\
&< \left(\frac{1}{n-1}\right)e^{\thmconst}\frac{(n-1)\thmconst}{\pi}\notag\\
&=\frac{\thmconst}{\pi}e^{\thmconst}\notag\\
&=\frac{2}{\pi}.
\end{align}
In particular, $\funcnorm{\Sdiff{\ut}}$ remains uniformly bounded away from $1$.
%
Thus, from Proposition~\ref{prop: c is A3s} and by the MTW maximum principle calculation in \cite[Section 4]{MTW05}, an {\em a priori} second derivative estimate for $u$ follows. Higher order estimates follow by the Evans-Krylov Theorem and standard elliptic theory, thus $I$ is closed.

To show openness, we set up the implicit function theorem as in
\cite[Theorem 17.6]{GT01}, by taking
\begin{equation*}
G:\left\{  u\in C^{2,\alpha}(S^{n}):\int udV_{\g}=0\right\}
\times\lbrack0,1\rbrack\rightarrow\left\{  v\in C^{0,\alpha}(S^{n}):\int
vdV_{\g}=0\right\}
\end{equation*}
to be defined as%
\begin{equation*}
G(u,t)=\rhoonet\left(  \frac{\det\left( u_{ij}+c_{ij}(\cdot,\Tplus[u]\right))
}{\det(-c_{i,j}(\cdot,\Tplus[u]))}\frac{\rhobart[{(\Tplus[u])}]}{\rhoonet}-1\right).
\end{equation*}
At a solution $u(x)$ at some time $t_{0},$ we have that
$G(u,t_{0})=0.$ Now the linearized operator on the first factor (whose principal part is a multiple of the operator $L$ in~\eqref{linearization}), is an elliptic operator with no zeroth
order terms. Since the linearized operator has index zero, the maximum
principle guarantees it is a bijection, and openness follows.
Since $e^{f_0}\equiv e^{g_0}$, we may take $u\equiv 0$ at $t=0$ and apply the continuity method to infer the existence of smooth solutions $u$ to~\eqref{cont eqn} for all $t\in [0, 1]$. Similarly, we obtain smooth solutions $v$ to~\eqref{rev eqn} for all $t\in[0, 1]$.

We now prove the $c$-convexity of $\ut$, solutions to~\eqref{cont eqn}. It is clear that the set $I':= \{t\in[0,1]\ \vert\ \ut\text{ is strictly }c\text{-convex}\}$ is relatively open and contains $0$. Now take any $t\in I'$. By~\eqref{Du est} we may apply Lemma~\ref{lemma: doesn't split} to find that $\csubdiff \ut(x)=\{\Tplus[\ut](x)\}$ for all $x\in\S^n$, that is, $\csubdiff \ut$ is a single valued map. The strict $c$-convexity of $\ut$ implies that $\csubdiff \ut$, hence $\Tplus[\ut]$ is injective. By~\eqref{cont eqn}, the Jacobian determinant of $\Tplus[\ut]$ is nonzero, so by an open-closed argument we see that $\Tplus[\ut]$ is surjective, and hence a bijection, in fact, a diffeomorphism. Thus we see that for any $y$,
\begin{equation*}
\ut(\Tplusinv[\ut](y))+c(\Tplusinv[\ut](y), y)=-(\ut)^c(y)
\end{equation*}
where $(\ut)^c$ is the $c$-transform from Definition~\ref{defin: c transform}, clearly differentiable by the above relation. Differentiating this relation twice and taking determinants of both sides, we see that $(\ut)^c$ satisfies equation~\eqref{rev eqn}. Thus, after normalizing $\vt$ by adding an appropriate constant, we find that $\vt=(\ut)^c$.

If $I'\neq[0,1]$, let $t_0:=\inf{([0,1]\setminus I')}>0$. Then $u_{t_0}$ is $c$-convex but not strictly $c$-convex. By the uniform convergence of $\ut$ and $\vt$ as $t\to t_0$, we can see that $v_{t_0}=u_{t_0}^c$. As above, $\csubdiff u_{t_0}^c(y)=\csubdiff v_{t_0}(y)=\{\Tplus[v_{t_0}](y)\}$ for all $y$, which implies in turn that $u_{t_0}$ is strictly $c$-convex, a contradiction. Thus, $I'=[0, 1]$ and $\ut$ is strictly $c$-convex for all $t\in[0, 1]$.

In particular, when $t=1$, $u:=\ut$ satisfies~\eqref{eqn2}, is $c$-convex, and $\csubdiff u(x)=\{\Tplus[u](x)\}$ for all $x\in\S^n$. Thus we may apply \cite[Theorem 5.10(ii) (replacing $u$ with $-u$)]{Vil09} to conclude that the measure $(\boldsymbol{Id}\times \Tplus[u])_\#\mu$ is a Kantorovich solution to the optimal transportation problem between $\mu$ and $\nu$. However, since $\Tplus[u]$ is a diffeomorphism, we see that it is actually a Monge solution. It is unique due to the uniqueness of the Kantorovich solution proven in \cite[Theorem 2.6]{GM00}.
\end{proof}

\begin{proof}[Proof of Theorem~\ref{thm: main 2}]
We follow the same method as the above proof, but instead of applying Theorem~\ref{thm: stay away property}, we use  Corollary~\ref{cor: wasserstein gradient bound on sphere} and the hypotheses to obtain the inequality~\eqref{Du est}. The remainder of the proof follows.
\end{proof}

\begin{proof}[Proof of Corollary~\ref{cor: main cor}]
By combining Theorem~\ref{thm: main 2} and Lemma~\ref{lemma: L^infty bounds wasserstein} we immediately obtain the claim.
\end{proof}
\bibliographystyle{plain}
\bibliography{mybiblio}

\end{document}